\documentclass[11pt,oneside]{amsart}

\title{On the classification of vertex-transitive structures}
\author{John Clemens}
\address{John Clemens, Boise State University, 1910 University Dr, Boise, ID 83725}
\email{johnclemens@boisestate.edu}
\author{Samuel Coskey}
\address{Samuel Coskey, Boise State University, 1910 University Dr, Boise, ID 83725}
\email{scoskey@nylogic.org}
\author{Stephanie Potter}
\address{Stephanie Potter, Boise State University, 1910 University Dr, Boise, ID 83725}
\email{stephaniepotter@boisestate.edu}

\usepackage[marginratio=1:1]{geometry}
\usepackage{setspace}
\usepackage{mathpazo,bm,amssymb,amsmath,mathrsfs,amsfonts,amsthm}
\usepackage{tikz}
\usepackage{enumerate}

\newtheorem{thm}{Theorem}[section]
\newtheorem{lem}[thm]{Lemma}
\newtheorem{prop}[thm]{Proposition}
\newtheorem{cor}[thm]{Corollary}
\theoremstyle{definition}
\newtheorem{defn}[thm]{Definition}
\theoremstyle{remark}

\newcommand{\Z}{\mathbb{Z}}
\newcommand{\Q}{\mathbb{Q}}
\newcommand{\N}{\mathbb{N}}
\newcommand{\F}{\mathbb{F}}
\newcommand{\Mod}{\text{Mod}}

\DeclareMathOperator{\Aut}{Aut}

\usepackage{etoolbox}
\makeatletter\pretocmd{\@seccntformat}{\S}{}{}
  \pretocmd{\@subseccntformat}{\S}{}{}\makeatother

\begin{document}
\begin{abstract}
  We consider the classification problem for several classes of countable structures which are ``vertex-transitive'', meaning that the automorphism group acts transitively on the elements. (This is sometimes called homogeneous.) We show that the classification of countable vertex-transitive digraphs and partial orders are Borel complete. We identify the complexity of the classification of countable vertex-transitive linear orders. Finally we show that the classification of vertex-transitive countable tournaments is properly above $E_0$ in complexity.
\end{abstract}
\maketitle

\section{Introduction}

In this article we study countable structures $A$ with the property that the automorphism group $\Aut(A)$ acts transitively on $A$. We will say that such structures are \emph{vertex-transitive}, or \emph{VT}. In \cite{John}, Clemens showed that the classification of countable vertex-transitive graphs is just as complex as the classification of arbitrary countable graphs. Since the classification of countable graphs is known to be of maximal complexity among classes of countable structures, the same is true of countable vertex-transitive graphs. We will extend Clemens' investigation to include the cases of countable directed graphs, partial orders, linear orders, and tournaments.

In order to describe our results, we briefly introduce Borel complexity theory, an area of logic which provides a framework to compare relative complexities of classification problems. In this theory we regard a classification problem as an equivalence relation on a standard Borel space. For a general class of examples, let $L$ be a countable relational language and consider the classification of countable $L$-structures up to isomorphism. The underlying standard Borel space is $\Mod(L)$, consisting of the $L$-structures with underlying set $\N$, and the classification problem may be identified with the isomorphism equivalence relation $\cong$ on $\Mod(L)$. For a particular example, to study the classification of countable graphs we take $L$ to consist of a single binary relation symbol. Then $\Mod(L)=2^{\N^2}$, and the space of countable graphs is the Borel subset $X$ consisting of just the symmetric, reflexive binary relations. The classification of countable graphs may be identified with the isomorphism equivalence relation restricted to $X$.

In order to compare complexities of various classification problems, we will make use of the notion of Borel reducibility. Here, given two equivalence relations $E$ and $F$ on standard Borel spaces $X$ and $Y$ respectively, one says that $E$ is \emph{Borel reducible} to $F$, denoted $E\leq_B F$, if there exists a Borel function $f\colon X\rightarrow Y$ such that for all $x,x'\in X$ we have $x\mathrel{E}x'$ if and only if $f(x)\mathrel{F}f(x')$. When this is the case, we say intuitively that the classification up to $E$-equivalence is no more complex than the classification up to $F$-equivalence.

There is a maximum possible complexity among isomorphism equivalence relations on classes of countable structures. We say that an equivalence relation $E$ is \emph{Borel complete} if and only if for every every countable language $L$, the isomorphism equivalence relation $\cong$ on $\Mod(L)$ is Borel reducible to $E$. For example, the isomorphism equivalence relation on the class of countable connected graphs is Borel complete. The result from \cite{John} mentioned above states that the isomorphism relation on countable connected vertex-transitive graphs is Borel complete as well.

We should note that the class of countable vertex-transitive graphs is \emph{not} a Borel subset of $2^{\N^2}$, and in fact it is not a standard Borel space. While this means the classification of countable vertex-transitive graphs does not fit within the classical Borel complexity theory, it is still possible to ask which equivalence relations are reducible to it, and therefore whether it is Borel complete. We will also briefly consider the use of absolutely $\bm{\Delta}^1_2$ reduction functions, and in this context it is sufficient to work with a $\bm{\Sigma}^1_1$ domain.

In the next section, we will use Clemens' reasoning to extend his result and show that isomorphism of countable directed vertex-transitive graphs is Borel complete, and isomorphism of countable vertex-transitive partial orders is Borel complete. In the third section, we classify the countable vertex-transitive linear orders, and in particular show that there are just $\omega_1$ many isomorphism classes of such orders. We also use our classification to provide a lower bound on the complexity of the isomorphism relation on such linear orders. In the last section, we study the isomorphism classification problem for countable vertex-transitive tournaments. We show that the complexity of this classification is properly more complex than $E_0$, the eventual equality relation on $2^\omega$. However, the question of whether or not it is Borel complete remains open.

We remark that some authors refer to vertex-transitive structures as ``homogeneous'', but we avoid the term since it is also often used to mean ultrahomogeneous. Without going into detail, ultrahomogeneity is a very strong property: the classification of countable ultrahomogeneous structures is always smooth (reducible to the equality relation on $2^\omega$), and hence not interesting from the point of view of Borel reducibility theory.

\textbf{Acknowledgement.} This work represents a portion of the third author's master's thesis \cite{steph_thesis}. The thesis was completed at Boise State University under the supervision of the second author, with significant input from the first author.

\section{Graphs and partial orders}

In this section we revisit the result of Clemens in \cite{John} which states that the isomorphism relation on the class of countable vertex-transitive graphs is Borel complete. We will use the details of the proof of this theorem to show that the isomorphism relations on the classes of countable vertex-transitive directed graphs and countable vertex-transitive partial orders are Borel complete too.

In this article, a \emph{directed graph} will always mean an oriented simple graph, so that there are no duplicate edges and no self-edges. We say that a directed graph is \emph{weakly connected} if the corresponding unoriented graph is connected.

\begin{thm}
  \label{directed graphs}
  The isomorphism relation on countable weakly-connected vertex-transitive directed graphs is Borel complete.
\end{thm}

\begin{proof}
  We will show that there exists a Borel reduction from countable graphs to countable weakly-connected vertex-transitive directed graphs. We will provide key details of the construction of the reduction function here, since we will need these details in the rest of the section. We will omit the proof that the construction yields a Borel reduction, instead describing how it may be extracted from \cite{John}.

  Let $G$ be a given countable graph and denote its vertices $\langle v_i\rangle_{i\in\N}$. We let $H$ be the group generated freely by the vertices of $G$ with the stipulation that adjacent vertices commute. That is, if $\F_\omega$ denotes the free group on generators $g_i$, we let $N$ be the normal subgroup of $\F_\omega$ generated by $\{g_ig_jg_i^{-1}g_j^{-1}\mid v_i\sim_Gv_j\}$ and define $H=\F_\omega/N$.

  Finally we form $\Gamma$, the directed Cayley graph of $H$ with generators $\langle g_i\rangle_{i\in\N}$. The vertices of $\Gamma$ are left cosets of $N$ in $\F_\omega$. We put a directed edge from $w_1N$ to $w_2N$ in $\Gamma$ if $g_iw_1N=w_2N$ for some $i$. Then $\Gamma$ is vertex-transitive because it is the Cayley graph of a group.

  We remark that given the graph $G$, it is possible to produce a code for the directed graph $\Gamma$ in a Borel fashion. This completes the construction of the reduction function $G\mapsto\Gamma$. As we noted previously, it remains to verify that $G_1\cong G_2$ if and only if $\Gamma_1 \cong \Gamma_2$. This is similar to \cite[Theorem~3.2]{John}, with the simplification that in the directed case we have no need for an extensionality hypothesis on the graphs $G$. The details may be found in \cite{steph_thesis}.
\end{proof}

In the next result we will use the above construction to show that the class of countable vertex-transitive partial orders is Borel complete. For this we recall that if $\Gamma$ is a directed graph, then its \emph{transitive closure} $C(\Gamma)$ is the directed graph obtained from $\Gamma$ by adding an edge $x\to y$ whenever there exists a directed path from $x$ to $y$ in $\Gamma$. The following fact is standard, though we remind the reader that in this article our directed graphs have no self-edges $x\to x$ nor bidirectional edges $x\to y\to x$.

\begin{prop}
  \label{transitive closure}
  If $\Gamma$ is a directed graph with no directed cycles, then $C(\Gamma)$ is a partial order with respect to the relation $x<y$ iff $x\to y$ in $C(\Gamma)$.
\end{prop}

We now arrive at the following result concerning vertex-transitive partial orders.

\begin{thm}
  \label{po}
  The isomorphism relation on countable vertex-transitive partial orders is Borel complete.
\end{thm}

\begin{proof}
  We again show that there is a Borel reduction from countable graphs to countable vertex-transitive partial orders. Given a countable graph $G$, we begin by constructing the directed graph $\Gamma$ from the proof of Theorem~\ref{directed graphs}.

  We first claim that $\Gamma$ has no directed cycles. Indeed, recall that $\Gamma$ is the directed Cayley graph of a group $H=\F_\omega/N$, where $N$ is generated by commutators of generators of $\F_\omega$. This implies that every word $w$ in $N$ has the property that the sum of the exponents of the generators appearing in $w$ is equal to $0$. On the other hand if $\Gamma$ contained a directed cycle, one would be able to find a word $w\in N$ such that the sum of the exponents of the generators in $w$ is positive. This establishes the claim.

  It follows from Proposition~\ref{transitive closure} that the transitive closure $P=C(\Gamma)$ of $\Gamma$ is a partial order. This completes the construction of the desired reduction function $G\mapsto P$.

  As we already argued, if $G_1\cong G_2$, then $\Gamma_1\cong\Gamma_2$. Since the transitive closure is isomorphism invariant, it follows that $C(\Gamma_1)\cong C(\Gamma_2)$ and hence that $P_1\cong P_2$.

  For the converse, we first claim that given a graph $P=C(\Gamma)$ as constructed above, it is possible to recover the set of $b\in\Gamma$ such that $N\rightarrow b$ in $\Gamma$, where $N$ denotes the vertex corresponding to the identity element of the Cayley graph.

  For this claim, we show that the out-neighbors of $N$ are exactly the $b\in\Gamma$ such that $N\rightarrow b$ in $C(\Gamma)$, and there does not exist a directed path of length greater than one from $N$ to $b$ in $C(\Gamma)$. Indeed, if $N\to b$ and there additionally exists a directed path from $N$ to $b$ of length greater than one, then we would have $b=g_jN\text{ and } b = g_{i_1}\dots g_{i_n}N$, and so in $H$ we would have $g_{i_1}\dots g_{i_n}g_j^{-1}=1$. This contradicts the fact that words in $N$ must have the sum of the exponents of all generators equal to $0$, and completes the claim.

  Now if $P_1\cong P_2$, then by definition $C(\Gamma_1)\cong C(\Gamma_2)$. Then it follows from the last claim that $\Gamma_1\cong\Gamma_2$. As we previously noted, the latter implies that $G_1\cong G_2$. This concludes the proof that $G\mapsto P$ is a reduction from countable graphs to countable vertex-transitive partial orders.
\end{proof}

\section{Linear orders}

In this section we give a complete characterization of the vertex-transitive linear orders. We then use this characterization to help describe the Borel complexity of the classification of countable vertex-transitive linear orders. We refer the reader to \cite{LO} for some of the basic linear order theory that will be used in this section.

We begin by introducing a key notion from linear order theory. If $L$ is any linear ordering and $x\in L$, we define the \emph{condensation class} of $x$ in $L$ by
\[c(x)=\left\{\,y\in L\mid\text{there are finitely many elements of $L$ between $x$ and $y$}\,\right\}.
\]
The sets $c(x)$ are convex and form an equivalence relation $\sim$ on $L$. The quotient linear ordering $L'=L/\mathord{\sim}$ is called the \emph{condensation} of $L$.

The condensation procedure can be iterated in a natural way. If $\alpha=\beta+1$ and $L^{(\beta)}$ has been constructed, we define
\[c^\alpha(x)=\left\{\,y\in L\mid\text{there are finitely many elements of $L^{(\beta)}$ between $c^\beta(x)$ and $c^\beta(y)$}\,\right\}.
\]
If $\alpha$ is a limit ordinal, we define
\[c^\alpha(x)=\bigcup_{\beta<\alpha}c^\beta(x)
\]
In either case, we may again define the corresponding equivalence relation $\sim_\alpha$, and then define the quotient ordering $L^{(\alpha)}=L/\mathord{\sim}_\alpha$.

We will use the terminology that a point $x$ in a linear order $L$ is \emph{left dense} if there exists a sequence in $L$ converging to $x$ from below, and \emph{left discrete} otherwise. The next two lemmas collect some of the information we will need to characterize the vertex-transitive linear orders.

\begin{lem}
  \label{dense}
  Let $L$ be a countable vertex-transitive linear ordering.
  \begin{enumerate}[(i)]
  \item If any point of $L$ is left or right dense, then $L\cong\Q$.
  \item If any point of $L$ is left or right discrete, then either $L=1$ or else every condensation class of $L$ is a copy of $\Z$.
  \end{enumerate}
\end{lem}

\begin{proof}
  (i) Assume without loss of generality that some point is left dense. Since $L$ is vertex-transitive, every point is left dense. Now let $a<b$ be given. Then there is an increasing sequence $\{b_n\}_{n\in\N}$ such that $\sup(b_n)=b$, and for some $n$ large enough we have $a<b_n<b$. We have thus shown that $L$ is a dense linear order, and hence isomorphic to either $\Q$, $\Q\cup\{\infty\}$, $\{-\infty\}\cup\Q$, or $\{-\infty\}\cup\Q\cup\{\infty\}$. Since $L$ is vertex-transitive, it must be the case that $L\cong\Q$. 

  (ii) Assume without loss of generality that some point is left discrete. Since $L$ is vertex-transitive, every point is left discrete. Additionally assume that $L\neq1$. Then by vertex-transitivity, $L$ has no least element. The last two statements imply that every condensation class of $L$ is nontrivial. Thus every condensation class is either finite and of size at least two, a copy of $\omega$, a copy of $\omega^\ast$, or a copy of $\Z$. Using vertex-transitivity one last time, no nontrivial condensation class may have a least or greatest element. It follows that every condensation class is a copy of $\Z$.
\end{proof}

We will say that a linear order $L$ is a \emph{condensation fixed point} if $c(x)=\{x\}$ for all $x\in L$. Of course if $L$ is vertex-transitive, then it is equivalent to say that $c(x)=\{x\}$ for some $x\in L$.

\begin{lem}
  \label{fp}
  Let $L$ be a countable vertex-transitive linear order. If $L$ is a condensation fixed point, then either $L=1$ or $L\cong\Q$.
\end{lem}

\begin{proof}
  By Lemma~\ref{dense}, the only possibilities are $L=1$, $L\cong\Q$, or every condensation class is a copy of $\Z$. In either of the first two cases, we are done. In the third case, we would clearly have that that $c(x)\neq\{x\}$ for every $x\in L$, which is contrary to the assumption that $L$ is a condensation fixed point.
\end{proof}

Next we will need the following very special class of linear orderings, the lexicographic powers of $\Z$.

\begin{defn}
  For any ordinal $\alpha$, we define the set
  \begin{align*}
    \Z^\alpha=\left\{s\colon\alpha\to\Z\mid \left\{\beta<\alpha:s(\beta)\not=0\right\}\text{ is finite}\right\}
  \end{align*}
  We equip $\Z^\alpha$ with the \emph{reverse lexicographic} ordering defined as follows. Given $s,t\in\Z^\alpha$ such that $s\neq t$, let $\mu$ be the greatest ordinal such that $s(\mu)\neq t(\mu)$, and let $s<t$ iff $s(\mu)<t(\mu)$.
\end{defn}

The powers of $\Z$ have the following recursive characterization. If $\alpha$ is any ordinal, then $\Z^{\alpha+1}\cong \Z^\alpha\cdot\Z$, that is, $\Z$ many copies of $\Z^\alpha$. And if $\lambda$ is any limit ordinal, then
\[\Z^\lambda\cong\left(\textstyle\sum_{\alpha<\lambda}\Z^\alpha\cdot\omega\right)^\ast+1+\textstyle\sum_{\alpha<\lambda}\Z^\alpha\cdot\omega.
\]
In other words, to the right of the middle $1$ we see $\omega$, followed by $\omega$ many copies of $\Z$, followed by $\omega$ many copies of $\Z^2$, and so on. And to the left of the middle $1$ we see the same thing backwards.

We are finally ready to state the characterization of the vertex-transitive linear orders.

\begin{thm}
  If $L$ is a vertex-transitive linear order, then there exists an ordinal $\alpha$ such that $L$ is isomorphic to either $\Z^\alpha$ or $\Z^\alpha\cdot\Q$.
\end{thm}

\begin{proof}
  If $L$ is a condensation fixed point, we are done by Lemma~\ref{fp}. Otherwise, $L$ has nontrivial condensation classes, and Lemma~\ref{dense} implies that every condensation class is a copy of $\Z$. Using this, it is easy to build an isomorphism $L\cong \Z\cdot L'$.

  Now $L'$ is again vertex-transitive, and so we may iterate the observation. Formally, by \cite[Theorem~5.9]{LO}, if $L$ is a linear order of cardinality $\kappa$, there exists an ordinal $\alpha<\kappa^+$ such that $L^{(\alpha)}$ is a condensation fixed point. We may further assume that $\alpha$ is the least such ordinal. This means that for every $\beta<\alpha$, the ordering $L^{(\beta)}$ has nontrivial condensation classes. Using the reasoning of the previous paragraph inductively, we can conclude that $L\cong\Z^\alpha\cdot L^{(\alpha)}$.

  Finally, since $L^{(\alpha)}$ is a condensation fixed point, Lemma~\ref{fp} implies that either $L^{(\alpha)}=1$ or $L^{(\alpha)}\cong\Q$. It follows that either $L\cong\Z^\alpha$ or $L\cong\Z^\alpha\cdot\Q$, as desired.
\end{proof}

It follows that the isomorphism relation on countable vertex-transitive linear orders has just $\omega_1$ many classes in any forcing extension, and so it is not Borel complete (in the sense that no Borel complete equivalence relation is Borel reducible to it). On the other hand we can use the above characterization to provide some additional information on the complexity of the classification of countable vertex-transitive linear orders. First we recall several definitions from descriptive set theory.

As we have mentioned, the set of countable vertex-transitive linear orders is not Borel. Thus in order to compare it with another well-studied equivalence relation, we use a more general notion than Borel reducibility. We say that a function is \emph{absolutely $\boldsymbol{\Delta}_2^1$} if it admits $\boldsymbol{\Pi}_2^1$ and $\boldsymbol{\Sigma}_2^1$ definitions which are equivalent in all forcing extensions.

The equivalence relation with which we will be comparing is equivalence of codes for countable ordinals. That is, we let $E_{\omega_1}$ denote the isomorphism equivalence relation on the set of well-ordered binary relations on $\N$.

\begin{thm}
  There exists an absolutely $\boldsymbol{\Delta}_2^1$ reduction from $E_{\omega_1}$ to the isomorphism relation on the set of countable vertex-transitive linear orders.
\end{thm}

\begin{proof}
  It suffices to show that there exists an absolutely $\boldsymbol{\Delta}^1_2$ function which maps a code for an ordinal $\alpha$ to a code for the linear ordering $\Z^\alpha$. This can be done by a recursive construction with the property that each step in the recursion is Borel.

  In detail, fix any binary relation $<_1$ with order type $\Z$. Given a code $<_\beta$ for the order type $\Z^\beta$, we can construct a code for $\Z^{\beta+1}=\Z^\beta\cdot\Z$ using the standard product construction. (Given natural numbers $n,m$, we write $n=\langle n_0,n_1\rangle$ and $m=\langle m_0,m_1\rangle$, where $\langle\cdot,\cdot\rangle$ is a pairing function. We then define $n<_{\beta+1}m$ if and only if $n_0 <_\beta m_0$ or $n_0=m_0$ and $n_1<_\beta m_1$.)

  Next, given a code for a limit ordinal $\lambda$, together with a $\lambda$-sequence of codes for $\Z^\beta$, $\beta<\lambda$, we can produce a code for $\Z^\lambda$. For this we use the previously mentioned property that $\Z^\lambda\cong\left(\textstyle\sum_{\alpha<\lambda}\Z^\alpha\cdot\omega\right)^\ast+1+\textstyle\sum_{\alpha<\lambda}\Z^\alpha\cdot\omega$, together with the natural constructions of ordinal-length products and sums.

  It is not difficult to see that both the successor step and the limit step described above may be carried out in a Borel fashion. It follows from this that one can construct the desired map in an absolutely $\boldsymbol{\Delta}_2^1$ fashion. For example, an infinite time Turing machine (ITTM) can easily be programmed to carry out the recursive construction, and ITTM-computable mappings are always absolutely $\boldsymbol{\Delta}_2^1$. (For the definition of ITTM and the statement of this fact, see \cite{ITTM}.)
\end{proof}

As we will discuss further in the next section, it follows from this result that there is no absolutely $\boldsymbol{\Delta}_2^1$ reduction from the isomorphism relation for vertex-transitive linear orders to any Borel equivalence relation.

\section{Tournaments}

In this section we study the classification of countable vertex-transitive tournaments. Recall that a \emph{tournament} is a directed graph with the property that for every pair $x,y$ of distinct vertices, there is either an edge $x\to y$ or an edge $y\to x$ and not both. In other words, a tournament is an oriented complete graph.

Since every linear order is a tournament with the edge relation $x\to y$ iff $x<y$, and every tournament is a directed graph, the work of the previous sections give lower and upper bounds on the complexity of the classification of countable vertex-transitive tournaments. The main result of this section gives an improvement on the lower bound.

Recall that $E_0$ is the equivalence relation defined on $2^\omega$ by $x\mathrel{E}_0y$ iff for all but finitely many $n$, $x(n)=y(n)$. We note that in the Borel reducibility hierarchy, $E_0$ lies properly above the equality relation $=$ on $2^\omega$, but ``just'' above in the sense that for any Borel equivalence relation $E$, either $E\leq_B\mathord{=}$ or else $E_0\leq_BE$ (see \cite{HKL}).

The equivalence relation $E_0$ has several other natural presentations which are equivalent up to Borel bireducibility. In particular we will use the equivalence relation $E_\Z$ defined on $2^\Z$ by $x\mathrel{E}_\Z y$ iff there exists $k$ such that $x(n+k)=y(n)$ for all $n$.

\begin{prop}
  For any comeager Borel subset $C$ of $2^\Z$, the restriction $E_\Z\restriction C$ is Borel bireducible with $E_0$.
\end{prop}

\begin{proof}
  It follows from \cite[Theorem~5.1]{DJK} that any equivalence relation which is given by the orbits of a $\Z$-action is Borel reducible to $E_0$. Thus $E_Z\upharpoonright{\mathcal{C}}$ is Borel reducible to $E_0$.
  
  Conversely, first note that $E_\Z$ has a dense orbit. It follows from \cite[Proposition~6.1.9]{Gao} that $E_\Z$ is \emph{generically ergodic}, meaning that every invariant Borel set is meager or comeager. Clearly, this implies $E_\Z\upharpoonright\mathcal{C}$ is generically ergodic as well. By \cite[Proposition 6.1.10]{Gao}, together with the fact that all equivalence classes of $E_\Z$ are meager, we obtain that $E_\Z\upharpoonright{\mathcal{C}}$ is not smooth. By the Glimm--Effros dichotomy \cite[Theorem~6.3.1]{Gao}, we conclude that $E_0$ is Borel reducible to $E_\Z\upharpoonright{\mathcal{C}}$.
\end{proof}

We are now ready to prove the main result of this section.

\begin{thm}
  There exists a Borel reduction from $E_0$ to the isomorphism relation on countable vertex-transitive tournaments.
\end{thm}

\begin{proof}
  By the proposition, it is sufficient to find a comeager subset $C\subset2^\Z$ and a Borel mapping $x\mapsto T_x$ from $C$ to countable vertex-transitive tournaments such that $x\mathrel{E}_\Z x'$ iff $T_x\cong T_{x'}$. We will begin by defining the mapping $x\mapsto T_x$ from $2^\Z$ to countable tournaments, and will define the appropriate comeager set $C$ midway through the proof.
  
  Given an element $x\in2^\Z$, we define a tournament $T_x$ on the vertex set $\Z\times\Z$ as follows. Given distinct vertices at positions $(m,n)$ and $(m',n')$ where $m,n,m',n'\in\Z$, we let $(m,n)\rightarrow (m',n')$ iff:
  \begin{itemize}
  \item $m=m'$ and $n>n'$; or
  \item $m'=m+1$ and $x(n'-n)=1$; or
  \item $m'\geq m+2$.
  \end{itemize}
  In all other cases we put an edge $(m',n')\rightarrow (m,n)$. Thus the digits of $x$ are coded into the edges between every pair of adjacent columns. See Figure~\ref{fig:edges} for an illustration.

  \begin{figure}[ht]
    \begin{center}
      \begin{tikzpicture}
        \node at (0,1.8) {$\vdots$};
        \node[shape=circle, fill=black, inner sep=0pt, minimum size=.4em, draw] (a) at (0,1) {};
        \node[shape=circle, fill=black, inner sep=0pt, minimum size=.4em, draw] (b) at (0,.5) {};
        \node[label=left:{$(0,0)$}, shape=circle, fill=black, inner sep=0pt, minimum size=.4em, draw] (c) at (0,0) {};
        \node[shape=circle, fill=black, inner sep=0pt, minimum size=.4em, draw] (d) at (0,-.5) {};
        \node at (0,-1.8) {$\vdots$};
        \node at (2,1.8) {$\vdots$};
        \node[shape=circle, fill=black, inner sep=0pt, minimum size=.4em, draw] (e) at (0,-1) {};
        \node[label=right:{$(1,2)$}, shape=circle, fill=black, inner sep=0pt, minimum size=.4em, draw] (f) at (2,1) {};
        \node[label=right:{$(1,1)$}, shape=circle, fill=black, inner sep=0pt, minimum size=.4em, draw] (g) at (2,.5) {};
        \node[label=right:{$(1,0)$}, shape=circle, fill=black, inner sep=0pt, minimum size=.4em, draw] (h) at (2,0) {};
        \node[label=right:{$(1,-1)$}, shape=circle, fill=black, inner sep=0pt, minimum size=.4em, draw] (i) at (2,-.5) {};
        \node[label=right:{$(1,-2)$}, shape=circle, fill=black, inner sep=0pt, minimum size=.4em, draw] (j) at (2,-1) {};
        \node at (2,-1.8) {$\vdots$};

        \draw[->,>=latex] (c) to (f);
        \draw[->,>=latex] (g) to (c);
        \draw[->,>=latex] (h) to (c);
        \draw[->,>=latex] (c) to (i);
        \draw[->,>=latex] (c) to (j);
      \end{tikzpicture}\\
      \caption{Coding the digits of $x$ into the edges between adjacent columns. In this figure we show just five edges between $(0,0)$ and the column to its right. Here we have $x(-2)=1$, $x(-1)=1$, $x(0)=0$, $x(1)=0$, and $x(2)=1$.\label{fig:edges}}
    \end{center}
  \end{figure}
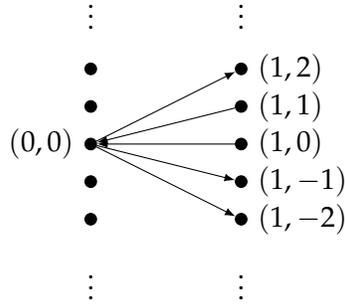

  It is not difficult to verify that the map $(a,b)\mapsto(a+m,b+n)$ preserves the definition of the edge relation in $T_x$ given above, and therefore $T_x$ is vertex-transitive. We now begin our verification that $x\mathrel{E}_\Z x'$ iff $T_x\cong T_{x'}$, though as noted above, this will only be true on a comeager set $C$ to be defined.

  First, suppose that $x\mathrel{E}_\Z x'$. That is, there exists some $k\in\Z$ such that, for every $n\in\Z$, $x(n)=x'(n+k)$. Then it is routine to verify using the definition that the map $\varphi(m,n)=(m,n+km)$ carries the edges of $T_x$ to the edges of $T_{x'}$, and hence witnesses that $T_x\cong T_{x'}$.


  For the other direction, we need to show that $T_x\cong T_{x'}$ implies $x\mathrel{E}_\Z x'$. Intuitively, we should be able to recover the shift-equivalence class of $x$ from the isomorphism class of $T_x$. To do so, we require $x$ and $x'$ have sufficiently many $0$ and $1$ values. To be precise, we let $C$ be the set consisting of all $z\in2^\Z$ satisfying the conditions:
  \begin{enumerate}[(i)]
  \item For every $n\neq0$ there exists some $k<n$ such that $z(k-n)=1$ and $z(k)=0$



  \item For every $n$ there exists some $k$ such that $z(-k)=0$ and $z(k-n)=0$

  \end{enumerate}
  Note that both of these conditions are $G_\delta$ and dense in $2^\Z$, and it follows that $C$ is comeager. Moreover if $z\in C$, then we can recover a number of properties of $T_z$ just from its isomorphism equivalence class. For instance, let $v$ be a vertex of $T_z$, and let $S_v$ be the set consisting of the column of $v$ together with the two columns to the left and two columns to the right of $v$.

  We claim that $S_v$ may be identified as the set of vertices that are involved in a three-cycle with $v$. To see this, suppose without loss of generality that $v=(0,0)$, and consider a second vertex $w$. In the case that $w=(0,n)$ is in the same column as $v$ with $n<0$ then condition~(i) ensures that there is a vertex $u=(1,k)$ such that $w\rightarrow u$ and $u\rightarrow v$. If $n>0$ we can exchange the roles of $v$ and $w$ to obtain the same conclusion. Condition~(i) can similarly be used if $w=(-1,n)$ or $w=(1,n)$, for $n\in\Z$. On the other hand, if $w=(-2,n)$ or $w=(2,n)$ then condition~(ii) ensures that there exists a third vertex $u$ such that $w\rightarrow u$ and $u\rightarrow v$.

  On the other hand if $w\notin S_v$ then $w$ cannot possibly be involved in a three-cycle with $v$. This follows from the fact that if an edge points ``left'' (i.e., from a higher-indexed column to a lower-indexed column) then its source and target are just one column apart. Thus if $w$ is at least three columns to the right of $v$, then one can have neither $w\to v$ nor $w\to u\to v$. And if $w$ is at least three columns to the left of $v$ then one can have neither $v\to w$ nor $v\to u\to w$. In each case $w$ is not involved in a three-cycle with $v$, completing the claim.
 
  Next, we can recover each of the five columns within $S_v$. Letting $C_{i,v}$ denote the column $i$ units to the right of $v$, we have:
  \begin{itemize}
  \item $C_{0,v}$ (the column of $v$) consists of all $w\in S_v$ such that $S_w=S_v$;
  \item $C_{-2,v}$ consists of all $w\in S_v$ such that there is no edge from $v$ to any element in $C_{0,w}$;
  \item $C_{2,v}$ consists of all $w\in S_v$ such that $v$ is in $C_{-2,w}$;
  \item $C_{-1,v}$ consists of all $w\in S_v$ such that $w\not\in C_{-2,v}$ and every $u\in C_{2,v}$ is not in $S_w$;
  \item $C_{1,v}$ consists of all vertices $w\in S_v$ such $v\in C_{-1,w}$.
  \end{itemize}

  Now from an isomorphic copy of $T_z$, we can recover $z$ up to shift equivalence as follows. Fix any $v$ and identify the column $C_{1,v}$. Then the $\to$ relation on $C_{1,v}$ is a linear order with order type $\Z$. We can thus identify $z$ up to shift equivalence simply by reading the edges between $v$ and the vertices $w\in C_{1,v}$.

  To conclude, suppose that $T_x\cong T_{x'}$. As outlined above, we may assume that both $x,x'\in C$. Let $\varphi$ be an isomorphism from $T_x$ to $T_{x'}$ and fix the vertex $v=(0,0)$. Using vertex-transitivity we may assume that $\varphi(v)=v$. Then using the fifth bullet point above, we have that $\varphi$ maps $C_{1,v}$ to $C_{1,v}$. Moreover, $\varphi$ preserves the $\Z$-order structure inherited from $\to$, and thus is simply a shift of $C_{1,v}$. Thus the edges from $v$ to vertices in $C_{1,v}$ in $T_x$ are a shift of the edges from $\varphi(v)$ to vertices in $C_{1,\varphi(v)}$ in $T_{x'}$. Since the digits of $x$ and $x'$ used to construct $T_x$ and $T_{x'}$ can be read from these edges, it must be the case that $x$ is a shift of $x'$. Thus we have $x\mathrel{E}_\Z x'$, as desired.
\end{proof}

Using the theorem together with the results of the previous section, we obtain the following.

\begin{cor}
  The isomorphism relation on countable vertex-transitive tournaments is strictly above $E_0$ in complexity.
\end{cor}

\begin{proof}
  First, recall the existence of an absolutely $\boldsymbol{\Delta}_2^1$ reduction from $E_{\omega_1}$ to the isomorphism relation on vertex-transitive linear orders. Since linear orders are tournaments, it follows that there is an absolutely $\boldsymbol{\Delta}_2^1$ reduction from $E_{\omega_1}$ to the isomorphism relation on vertex-transitive tournaments.

  However, there is no absolutely $\boldsymbol{\Delta}_2^1$ reduction from $E_{\omega_1}$ to any Borel equivalence relation. Indeed, if there were such a reduction $f$, then it would be possible to find an absolutely $\bm{\Delta}^1_2$ injection $F$ from codes for ordinals to codes for sets of reals of bounded Borel rank. (In fact one can take $F(x)$ to be a code for $[f(x)]_E$.) However, this contradicts the remark at the end of \cite[\S3]{hjorth}, which states that no such mapping exists.
  
  It follows from this that the isomorphism relation for vertex-transitive graphs is properly above $E_0$ in complexity, as claimed.
\end{proof}

We close with the question of whether there exists a Borel reduction from a Borel complete equivalence relation to the isomorphism relation on countable vertex-transitive tournaments.

\bibliographystyle{alpha}
\bibliography{bib}

\end{document}